\documentclass[12pt]{amsart}

\usepackage{amssymb,amsmath}

\theoremstyle{definition}
\newtheorem{theorem}{Theorem}
\newtheorem{lemma}{Lemma}

\theoremstyle{remark}

\begin{document}

\title{On Whitney Embedding of O-minimal Manifolds}

\author[R. Bianconi]{Ricardo Bianconi}
\address{Departamento de Matem\'atica, Instituto de Matem\'atica e Estat\'is\-tica, Universidade de S\~ao Paulo, Rua do Mat\~ao 1010, CEP 05508-090, S\~ao Paulo, SP, BRAZIL.}
\email{bianconi@ime.usp.br}

\author[R. Figueiredo]{Rodrigo Figueiredo}
\thanks{The second author was supported by the research grant by CNPq Brazil, Proc. No. 141829/2014-1.}
\address{Departamento de Matem\'atica, Instituto de Matem\'atica e Estat\'is\-tica, Universidade de S\~ao Paulo, Rua do Mat\~ao 1010, CEP 05508-090, S\~ao Paulo, SP, BRAZIL.}
\email{rodrigof@ime.usp.br}

\author[R. A. Figueiredo]{Robson A. Figueiredo}
\address{Departamento de Matem\'atica, Instituto de Matem\'atica e Estat\'is\-tica, Universidade de S\~ao Paulo, Rua do Mat\~ao 1010, CEP 05508-090, S\~ao Paulo, SP, BRAZIL.}
\email{robs@ime.usp.br}

\begin{abstract}
We prove a definable version of the Whitney embedding theorem for abstract-definable $\mathcal{C}^p$ manifolds with $1\leq p<\infty$, namely: every abstract-definable $\mathcal{C}^p$ manifold is abstract-definable $C^p$ embedded into $R^N$, for some positive integer $N$. As a consequence, we show that every abstract-definable $\mathcal{C}^p$ manifold has a compatible $\mathcal{C}^{p+1}$ atlas.
\end{abstract}

\keywords{o-minimal, abstract definable manifolds, embedding, smoothing}
\subjclass[2010]{03C64, 12D15, 14Pxx}

\maketitle

\section{Introduction} 

O-minimal structures generalize the notion of semialgebraic sets and have been very successful recently in their applications, mainly in arithmetic geometry, \cite{pila2011,pila-wilkie2006,gao2017,klinger-ullmo-yafaev2016}.

We deal here with Whitney's Embedding Theorem for manifolds definable in the context of o-minimal expansions of a real closed field.

In the mid-1980s, M. Shiota introduced the notion of an abstract Nash manifold of class $\mathcal{C}^p$, \cite{shiota1986}, and proved that every abstract Nash manifold of class $\mathcal{C}^p$, with $1\leq p<\infty$, can be $\mathcal{C}^p$ Nash embedded into some Euclidean space. Roughly speaking, an abstract Nash manifold of class $\mathcal{C}^p$ is a topological manifold equipped with a finite atlas, whose transitions maps are Nash $\mathcal{C}^p$ diffeomorphisms. The method M. Shiota used to prove his embedding theorem differs from that usually employed in the proof of  Whitney's Embedding Theorem (see Theorem 6.15, \cite{johnmlee}), since the underlying field is not necessarily Archimedean, and it can be adjusted to the o-minimal setting. In this direction, T. Kawakami shows, in our parlance, that every $n$-dimensional abstract definable $\mathcal{C}^p$ manifold, with $2\leq p<\infty$, is abstract-definably $\mathcal{C}^p$ embeddable into $\mathbb{R}^{2n+1}$, where the fixed structure is an o-minimal expansion of the real field. He obtains such an embedding by means of the fact that every definable $\mathcal{C}^p$ function can be approximated by an injective definable $\mathcal{C}^p$ immersion in the $\mathcal{C}^p$ topology and the fact that every affine abstract definable $\mathcal{C}^p$ manifold is either compact or abstract-definably $\mathcal{C}^p$ diffeomorphic to the interior of some compact abstract definable $\mathcal{C}^p$ manifold with boundary. With a fixed o-minimal expansion of a real closed field as the underlying structure, we prove in the present work that every abstract-definable $\mathcal{C}^p$ manifold can be abstract-definably $\mathcal{C}^p$ embedded into some Euclidean space, where $1\leq p<\infty$ (Theorem \ref{wet}), by following \cite{shiota1986} and \cite{kawakami2002}. It is worth mentioning that, in the same setting as ours, A. Berarducci and M. Otero established a Whitney's Embedding Theorem for the case of definably compact abstract definable $\mathcal{C}^p$ manifolds (Theorem 10.7, \cite{berarducci-otero2001}). Hence, the first main result of this paper (Theorem \ref{wet}) generalizes all of these cited previous works.

By virtue of the Whitney's Embedding Theorem established in the first part of the paper for the category of abstract definable $\mathcal{C}^p$ manifolds, whose fixed structure is an o-minimal expansion of a real closed field $R$, we can view an abstract definable $\mathcal{C}^p$ manifold as a definable $\mathcal{C}^p$ submanifold of some $R^N$. We then use a theorem on smoothing definable submanifolds by J. Escribano (Theorem 1.11, \cite{escribano2002}) to obtain our second main result of this paper (Theorem \ref{smo}). Theorem \ref{smo} has particular interest for the de Rham cohomology of abstract-definable manifolds, since it allows us to construct cochain complex of abstract definable $\mathcal{C}^p$ forms, and thereby establish a de Rham-like cohomology for abstract-definable $\mathcal{C}^p$ manifolds, with $1\leq p<\infty$, just like it has been settled in \cite{figueiredo2017} for the category of abstract definable $\mathcal{C}^\infty$ manifolds, whose underlying o-minimal structure has the additional assumption of admitting smooth cell decomposition.

\section{Preliminaries} 

We recall some definitions and facts.
\\

An \textit{o-minimal expansion $\mathcal{R}$ of a real closed field} $R$ is a family  $\{\mathcal{R}_n\}_{n\in\mathbb{N}}$ of subsets of $R^n$, such that
\begin{enumerate}
\item
each $\mathcal{R}_n$ is a Boolean algebra of subsets of $R\sp{n}$;
\item
$R\in \mathcal{R}_1$, and
the graphs of the sum and the product of $R$ belong to $\mathcal{R}_3$;
\item
if $A\in \mathcal{R}_n$ and $B\in \mathcal{R}_m$ then $A\times B\in \mathcal{R}_{n+m}$;
\item
if $T:R\sp{m}\to R\sp{n}$ is an $R$-linear transformation, and $A\in \mathcal{R}_m$, then $T(A)\in \mathcal{R}_n$;
\item
(o-minimality) the only sets in $\mathcal{R}_1$ are the unions of finitely many points and open intervals with endpoints in $R\cup\{\pm\infty\}$.
\end{enumerate}
We say that a subset $A\subseteq R^n$ is \textit{definable} (in $\mathcal{R}$) if $A\in \mathcal{R}_n$. A map $f\colon A\to B$, with $A\subseteq R^m$ and $B\subseteq R^n$, is called \textit{definable} (in $\mathcal{R}$) if its graph is definable. A subset $X\subseteq R^n$ is said to be \textit{definable in $\mathcal{R}$ with parameters in $A\subseteq R$} if $X$ is a fiber $Y_a\mathrel{\mathop:}=\{x\in R^n\,:\, (x,a)\in Y\}$ of a definable set $Y\subseteq R^{n+m}$ above the $m$-tuple $a\in A^m$.   

We refer the reader to \cite{vddries98} and \cite{vandendries-miller1996} for a thorough introduction to o-minimal structures.

Throughout the paper, $\mathcal{R}$ denotes a fixed but arbitrary o-minimal expansion of a real closed field $R$, and by ``definable'' we mean ``definable in $\mathcal{R}$ with parameters in $R$'', unless otherwise stated. 

Let $M$ be a set, and let $\{\phi_i\colon U_i\to \phi_i(U_i)\subseteq R^m\}_{i\in I}$ be a finite family of set-theoretic bijections, where each $U_i$ is a subset of $M$ and $\phi_i(U_i)$ is a definable open set in $R^m$. Such a collection is said to be an \textit{abstract-definable $\mathcal{C}^p$ atlas} on $M$ of dimension $m$, where $0\leq p<\infty$, if $M=\bigcup_{i\in I} U_i$ and for any $i,j\in I$ the sets  $\phi_i(U_i\cap U_j), \phi_j(U_i\cap U_j)$ are definable and open in $R^m$, and the \textit{transition maps} $\phi_j\circ \phi_i^{-1}\colon \phi_i(U_i\cap U_j) \to \phi_j(U_i\cap U_j)$ are definable $\mathcal{C}^p$ diffeomorphisms.

The relation $\sim$ defined on the set of all abstract-definable $\mathcal{C}^p$ atlases of dimension $m$ on a set $M$, given by $\mathcal{A}\sim \mathcal{B}$ if and only if $\mathcal{A}\cup \mathcal{B}$ is an abstract-definable $\mathcal{C}^p$ atlas on $M$, is an equivalence relation. In this case, we say that $\mathcal{A}$ and $\mathcal{B}$ are \textit{compatible}.

Any abstract-definable $\mathcal{C}^p$ atlas $\{\phi_i\colon U_i\to \phi_i(U_i)\subseteq R^m\}_{i\in I}$ on a set $M$ endows such a set with a topology whose open sets are those subsets $U\subseteq M$ such that $\phi_i(U_i\cap U)$ are open in $R^m$ for all $i\in I$. This is the unique topology on $M$ in which each $U_i$ is open and every $\phi_i$ is a homeomorphism. Two equivalent abstract-definable $\mathcal{C}^p$ atlases on a set induce the same topology, the \textit{manifold topology}. The manifold topology is $T_1$, although is not Hausdorff.

An \textit{abstract-definable $\mathcal{C}^p$ manifold} of dimension $m$ is a set $M$ together with a $\sim$-equivalence class of $m$-dimensional abstract-definable $\mathcal{C}^p$ atlases on $M$, whose manifold topology is Hausdorff. 

Let $(M,\mathcal{A})$ and $(N,\mathcal{B})$ be two abstract-definable $\mathcal{C}^p$ manifolds. A subset $A\subseteq M$ is called an \textit{abstract-definable set} in $M$ if $\phi(U\cap A)$ is definable for every chart $(U,\phi)$ in $\mathcal{A}$. A map $f\colon M\to N$ is said to be \textit{abstract-definable} (resp., \textit{abstract-definable $\mathcal{C}^p$}, an \textit{abstract-definable $\mathcal{C}^p$ diffeomorphism}) if for every point $x\in M$ and any charts $(U,\phi)\in \mathcal{A}$, $(V,\psi)\in \mathcal{B}$ with $x\in U$ and $f(x)\in V$ the map 
$$
\psi\circ f\circ \phi^{-1}|_{\phi(U\cap f^{-1}(V))}\colon \phi(U\cap f^{-1}(V))\to \psi(f(U)\cap V)
$$ 
is definable (resp., a $\mathcal{C}^p$ map, a definable $\mathcal{C}^p$ diffeomorphism). (See \cite{vddries98}, pp. 114-116, for the notion of a $\mathcal{C}^p$ map.) The set of all abstract-definable open sets in $M$ forms a basis for the manifold topology. 

Fix $x\in M$ and consider the set $\mathcal{C}^p(x)$ of all abstract-definable $\mathcal{C}^p$ maps $\alpha\colon I\to M$, where $I$ is an open interval containing $0$, such that $\alpha(0)=x$, on which we have the equivalence relation 
$$
\alpha_1\sim \alpha_2\Leftrightarrow (\phi\circ \alpha_1)'(0)=(\phi\circ \alpha_2)'(0),
$$
for some chart $(U,\phi)$ on $M$ at $x$. By virtue of the chain rule, we may replace the condition ``for some chart on $M$ at $x$'' with ``for any chart on $M$ at $x$'' in the definition  of $\sim$. The quotient set $\mathcal{C}^p(x)/\sim$ is denoted by $T_x M$.  

If $(U,\phi)$ is a chart on $M$ at $x$, the induced map $\Phi_x\colon T_x M\to R^m$ defined as $[\alpha]\mapsto (\phi\circ \alpha)'(0)$ is bijective, and hence there is a unique $R$-vector space structure on $T_x M$ which makes $\Phi_x$ into a linear isomorphism, namely: $v+w=\Phi_x^{-1}(\Phi_x(v)+\Phi_x(w))$ and $rv=\Phi_x^{-1}(r\Phi_x(v))$, for $v,w\in T_x M, r\in R$. These operations are independent of the choice of $(U,\phi)$. The set $T_x M$ together with such a linear structure is called \textit{tangent space} to $M$ at $x$ and its elements are the \textit{tangent vectors} to $M$ at $x$.

An abstract-definable $\mathcal{C}^p$ map $f\colon M\to N$ induces at each point $x\in M$ a linear map $d_x f\colon T_x M\to T_{f(x)}N$, the \textit{differential} of $f$ at $x$, by setting $d_x f ([\alpha])=[f\circ \alpha]$. Under the identification $T_x R^m\equiv R^m$, we obtain $d_x \phi=\Phi_x$.

Given a chart $(U,\phi)$ at a point $x\in M$, the set $\{\partial/\partial x^1|_x,\ldots,\partial/\partial x^m|_x\}$ forms a basis for $T_x M$, where $\partial /\partial x^i|_x$ is $(d_x \phi)^{-1}(e_i)$ and $e_i$ denotes the $i$th standard basis vector $(0,\ldots,1,\ldots,0)$ of $R^m$. Hence, a tangent vector $X_x\in T_x M$ can be uniquely written as $X_x=\sum_{i=1}^ma_i(\partial/\partial x^i|_x)$, with $(a_1,\ldots,a_m)\in R^m$. If $X_x=[\alpha]$, for some $\alpha\in \mathcal{C}^p(x)$, then $(a_1,\ldots,a_m)=(\phi\circ \alpha)'(0)$. 

Let $f\colon M\to R$ be an abstract-definable $\mathcal{C}^p$ function. The \textit{directional derivative} $X_x f$ of $f$ at $x\in M$ is defined to be $(f\circ \alpha)'(0)$. If $(U,\phi)$ is a chart at $x$, then applying the chain rule to $(f\circ \phi^{-1}\circ \phi\circ \alpha)'(0)$ we get  $X_x f=\sum_{i=1}^ma_i(\partial(f\circ \phi^{-1})/\partial r^i) (\phi(x))$, where $a_i$ are the components of $X_x$ in the basis $\{\partial/\partial x^i|_x\}_i$. Particularly, $(\partial/\partial x^i|_x)f=(\partial(f\circ \phi^{-1})/\partial r^i)(\phi(x))$.

A map $f\colon M\to N$ between abstract definable $\mathcal{C}^p$ manifolds is said to be an \textit{abstract-definable $\mathcal{C}^p$ immersion} if for each $x\in M$ the differential $d_x f\colon T_x M\to T_{f(x)} N$ of $f$ at $x$ is injective. If, in addition, $f$ is a homeomorphism onto its image then it is called an \textit{abstract-definable $\mathcal{C}^p$ embedding}.

\section{Embedding of Abstract Definable Manifolds} 

\begin{theorem}\label{wet}
Any abstract definable $\mathcal{C}^p$ manifold is abstract-definable $\mathcal{C}^p$ embedded into $R^n$, for some $n$.
\end{theorem}
\begin{proof}
Let $M$ be an abstract definable $\mathcal{C}^p$ manifold, with $\mathcal{C}^p$ atlas $\{h_i\colon U_i\to h_i(U_i)\,|\, i=1,\ldots,k\}$. By Proposition 4.22 (\cite{vandendries-miller1996}), for each $i\in I$, there is a definable $\mathcal{C}^p$ function $\phi_i\colon R^m\to R$ such that $\phi_i^{-1}(0)=\overline{h_i(U_i)}\setminus h_i(U_i)$. Define the map $h'_i\colon U_i\to R^{m+1}$ to be the rule 
$$
x\mapsto (h_i(x),1/\phi_i(h_i(x)))
$$
Note that the image $\text{Im}\,h'_i$ of $h'_i$ is the graph of $1/\phi_i$ restricted to $h_i(U_i)$. In particular, $\text{Im}\, h_i'$ is definable.
\\

\noindent\textbf{Claim 1}. $\text{Im}\, h'_i$ is a definable closed subset of $R^{m+1}$.
\\

\noindent\textit{Proof of Claim 1}. Let $(a,b)$ be an arbitrary point in $R^{m+1}\setminus \text{Im}\, h'_i$. Thus, we have two cases: $a\in h_i(U_i)$ and $a\not\in h_i(U_i)$. Suppose first $a\not\in h_i(U_i)$. If $a\in \overline{h_i(U_i)}\setminus h_i(U_i)$, then $\phi_i(a)=0$. If $b=0$, then by taking $N\mathrel{\mathop:}=\max\{|\phi_i(x)|\,:\,x\in \overline{h_i(U_i)}\}>0$ (recall that $\phi_i$ is definable continuous!) and a small neighborhood $V$ of $a$, it follows that $V\times ]-N,N[$ is an open set containing $(a,b)$ disjoint from the graph of $1/\phi_i$. On the other hand, if $b\neq 0$ then by the continuity of $\phi_i$ there is a neighborhood $V$ of $a$ such that $|\phi_i|_V|<1/\epsilon$ for $\epsilon>|b|$. Hence, $V\times ]-\epsilon,\epsilon[$ is a neighborhood of $(a,b)$ contained in $R^{m+1}\setminus \text{Im}\, h'_i$. Now, assume that $a\not \in \overline{h_i(U_i)}$. Then, $V\times R$ is a neighborhood of $(a,b)$ disjoint from the graph of $1/\phi_i$, where $V\mathrel{\mathop:}=R^m\setminus \overline{h_i(U_i)}$. Suppose now $a\in h_i(U_i)$. Since $R^m$ is definably regular, there is a definable open set $V$ containing $a$ with $\overline{V}\subseteq h_i(U_i)$. Because the graph of $(1/\phi_i)|_{\overline{V}}$ is closed in $R^{m+1}$ and does not contain $(a,b)$, there is an open subset $W\subseteq R$ such that, shrinking $V$ if necessary, $V\times W$ does not intersect $\text{Im}\, h_i'$.
\\

\noindent\textbf{Claim 2}. $h_i'$ is definably proper.
\\

\noindent\textit{Proof of Claim 2}. 
It suffices to prove that for any definably compact nonempty subset $K\subseteq R^{m+1}$, and any abstract-definable continuous curve $\gamma\colon ]a,b[\to U_i$ contained in $h_i'^{-1}(K)$, the limits $\lim_{t \to a^+}\gamma(t)$ and $\lim_{t\to b^-}\gamma(t)$ belong to $h_i'^{-1}(K)$. Indeed, for any abstract-definable continuous curve $\gamma\colon ]a,b[\to U_i$ contained in $h_i'^{-1}(K)$, $h'_i\circ \gamma\colon ]a,b[\to R^{m+1}$ is an abstract-definable continuous curve contained in $\text{Im}\, h_i'\cap K$. By hypothesis, $L_1\mathrel{\mathop:}=\lim_{t\to a^+}(h_i'\circ \gamma)(t),L_2\mathrel{\mathop:}=\lim_{t\to b^-}(h_i'\circ \gamma)(t)\in K$, and in view of Claim 1 both $L_1$ and $L_2$ are in $\text{Im}\, h'_i\cap K$. Hence, the limits $\lim_{t\to a^+}\gamma(t)=h_i'^{-1}(L_1)$ and $\lim_{t\to b^-} \gamma(t)=h_i'^{-1}(L_2)$ belong to $h_i'^{-1}(K)$.
\\

\noindent\textbf{Claim 3}. $h_i'$ is an abstract-definable $\mathcal{C}^p$ embedding from $U_i$ into $R^{m+1}$.
\\

\noindent\textit{Proof of Claim 3}. Since $h_i$ is an abstract-definable $\mathcal{C}^p$ diffeomorphism, $h_i'$ is an abstract-definable $\mathcal{C}^p$ immersion. The fact that $h_i'$ is a homeomorphism from $U_i$ onto $h_i'(U_i)$ follows from being the composition of two homeomorphisms, namely $h_i$ and $y\mapsto (y,1/\phi_i(y))\colon h_i(U_i)\to \text{Graph}(1/\phi_i)$. 
\\

Denote by $\pi$ the stereographic projection from $S^{m+1}\setminus \{N\}$ onto $R^{m+1}$, where $N$ stands for the north pole in $S^{m+1}$. Since $\pi$ is a definable $\mathcal{C}^p$ diffeomorphism, the map $h_i''\colon U_i\to R^{m+2}$, given by 
$$
h''_i\mathrel{\mathop:}= \pi^{-1}\circ h'_i,
$$
is an abstract-definable $\mathcal{C}^p$ (definably proper) embedding, whose image is bounded and such that $\overline{h''_i(U_i)}\setminus h''_i(U_i)=\{N\}$. To see the last assertion, first note that $||h'_i||$ is unbounded. Consequently, $N\in \overline{(\pi^{-1}\circ h'_i)(U_i)}$. Moreover, if towards a contradiction there exists a point $P$ in $\overline{h''_i(U_i)}\setminus h''_i(U_i)$ distinct from $N$, then $P$ lies in the domain of $\pi$, and $\pi(P)\in \overline{h'_i(U_i)}=h_i(U_i)$. Thus, $P\in \pi^{-1}(h'_i(U_i))=h''_i(U_i)$, contradicting the assumption.

Let $\psi\colon R^{m+2}\to R^{m+2}$ be the map given by 
$$
\psi(r_1,\ldots,r_{m+2})\mathrel{\mathop:}=\sum\limits_{j=1}^{m+2} r_j^{2l}(r_1,\ldots,r_{m+2}),
$$
for a sufficiently large $l$ and define $g_i\colon U_i\to R^{m+2}$ by 
$$
g_i\mathrel{\mathop:}= \psi\circ h''_i.
$$
The map $g_i$ has the same properties as $h_i''$.

We now extend $g_i$ to $M$ by the north pole, that is, let $\widetilde{g}_i\colon M\to R^{m+2}$ be given by the rule
\[
    \widetilde{g}_i\mathrel{\mathop:}= 
\begin{cases}
    g_i& \text{in } U_i\\
    N&  \text{in } M\setminus U_i
\end{cases}
\]
It is not hard to see that $\widetilde{g}_i$ is abstract-definable $\mathcal{C}^p$, and consequently so is the map $g\colon M\to R^{k(m+2)}$ given by
$$
g\mathrel{\mathop:}=(\widetilde{g}_1,\ldots,\widetilde{g}_k).
$$

\end{proof}

\section{Smoothing of Abstract Definable Manifolds} 

\begin{lemma}[Definable local immersion theorem]\label{1s}
Let $f\colon U\to R^{m+n}$ be a definable $\mathcal{C}^p$ map, where $U$ is a definable open subset of $R^m$ and $p\geq 1$. Suppose the differential $d_af\colon R^m\to R^{m+n}$ of $f$ at $a\in U$ is one-to-one. Then, there exist definable open subsets $V\subseteq U$, $W\subseteq R^n$ and $Z\subseteq R^{m+n}$ with $a\in V$, $0\in W$ and $f(a)\in Z$, and a definable $\mathcal{C}^p$ diffeomorphism $h\colon Z\to V\times W$ such that $(h\circ f)(r_1\ldots,r_m)=(r_1,\ldots,r_m,0,\ldots,0)$ for all $(r_1,\ldots,r_m)\in V$.
\end{lemma}
\begin{proof}
Denote by $E$ the vector subspace $d_af(R^m)$ of $R^{m+n}$, and choose a vector subspace $F$ such that $R^{m+n}=E\oplus F$. Given a basis $\{v_1,\ldots,v_n\}$ for $F$, take $\widetilde{f}\colon U\times R^n\to R^{m+n}$ to be the map defined as 
$$
\widetilde{f}(x,y)\mathrel{\mathop:}=f(x)+\sum_{i=1}^ny_iv_i,
$$ 
where $y=(y_1,\ldots,y_n)$. Clearly, $\widetilde{f}$ is a definable $\mathcal{C}^p$ map. Moreover, the linear map $d_{(a,0)}\widetilde{f}\colon R^m\times R^n\to E\oplus F$ is given by the rule 
$$
d_{(a,0)}\widetilde{f}(u,w)=d_af(u)+\sum_{i=1}^nw_iv_i,
$$ 
where $u\in R^m$ and $w=(w_1,\ldots,w_n)\in R^n$, thereby is injective, and ultimately is a linear isomorphism. From the definable inverse function theorem (see for instance Theorem 7.2.11, \cite{vddries98}) it follows that there exist definable open neighborhoods $\Omega\subseteq U\times R^n$ of $(a,0)$ and $\Omega'\subseteq R^{m+n}$ of $\widetilde{f}(a,0)$ such that $\widetilde{f}\colon \Omega\to \Omega'$ is a definable $\mathcal{C}^p$ diffeomorphism. By recalling that the topology on $U\times R^n$ is the product topology, we may pick definable open sets $V\subseteq U$ and $W\subseteq R^n$ with $(a,0)\in V\times W\subseteq \Omega$. If we put $Z\mathrel{\mathop:}=\widetilde{f}(V\times W)$ and $h\mathrel{\mathop:}=\widetilde{f}^{-1}\colon Z\to V\times W$, the result thus follows.
\end{proof}

\begin{lemma}[Local immersion theorem for abstract definable manifolds]\label{2s}
Let $M$ and $N$ be abstract definable $\mathcal{C}^p$ manifolds of dimension $m$ and $n$, respectively,  and let $f\colon M\to N$ be an abstract definable $\mathcal{C}^p$ immersion. Then for any point $x\in M$ there exist charts $(\phi,U)$ over $M$ and $(\psi,V)$ over $N$, with $x\in U$ and $f(U)\subseteq V$, such that  
$$
(\psi\circ f \circ \phi^{-1})(r_1,\ldots,r_m)=(r_1,\ldots,r_m,0,\ldots,0)\in R^n
$$ 
for all $(r_1,\ldots,r_m)\in \phi(U)$.   
\end{lemma}
\begin{proof}
Let $(\phi_1,U_1)$ be an arbitrary chart on $M$ with $x\in U_1$ and and $(\psi_1, V_1)$ a chart on $N$ with $f(x)\in V_1$. Since $f_1$ is abstract definable continuous, $U_2\mathrel{\mathop:}=U_1\cap f^{-1}(V_1)$ is an abstract definable neighborhood of $x$ with $f(U_2)\subseteq V_1$. Moreover, the restriction $\phi_2\mathrel{\mathop:}=\phi_1|_{U_2}\colon U_2\to \phi_1(U_2)$ is a chart, which is compatible with the abstract definable $\mathcal{C}^p$ atlas on $M$. By definition, the representation $\widetilde{f}\mathrel{\mathop:}=\psi_1\circ f\circ \phi_2^{-1}$ of $f$ with respect to the charts $\phi_2$ and $\psi_1$ is a definable  $\mathcal{C}^p$ map. From the chain rule and the fact that $d_xf$ is injective, it follows that the linear map $d_{\phi_2(x)}\widetilde{f}=d_{f(x)}\psi_1\circ d_xf\circ d_{\phi_2(x)}\phi_2^{-1}$ is also injective. By Lemma \ref{1s}, there exists a definable open set $\widetilde{U}\subseteq \phi_2(U_2)$ containing $\phi_2(x)$ and a definable $\mathcal{C}^p$ diffeomorphism $h\colon Z\to Z'$ where $Z,Z'\subseteq R^n$ are definable open sets with $\widetilde{f}(\widetilde{U})\subseteq Z$ such that 
$$
h(\widetilde{f}(r_1,\ldots,r_m))=(r_1,\ldots,r_m,0,\ldots,0)\in R^n
$$
for all $(r_1,\ldots,r_m)\in \widetilde{U}$. Now, take $U$ to be the abstract definable open subset $\phi_2^{-1}(\widetilde{U})\subseteq U_2$, $\phi$ to be the map $\phi_2|_U\colon U\to \widetilde{U}$, $V$ to be the abstract definable open subset $\psi_1^{-1}(Z\cap \psi_1(V_1))\subseteq V_1$ and $\psi$ to be the map $(h\circ \psi_1)|_V\colon V\to Z\cap \psi_1(V_1)$. The proof is done by noticing that $h\circ \widetilde{f}=\psi\circ f\circ \phi$ on $\phi(U)$.
\end{proof}

A definable set $X\subseteq R^n$ is called a \textit{definable $\mathcal{C}^p$ submanifold of dimension $m$} of $R^n$ if for every point $x\in X$ there exist definable open sets $U,V\subseteq R^n$, with $x\in V$ and $0\in U$, and a definable $\mathcal{C}^p$ diffeomorphism $\phi\colon U\to V$ such that $\phi(0)=x$ and $\phi(R^m\cap U)=V\cap X$. By replace ``semialgebraic'' with ``definable'', and ``Nash submanifold'' with ``definable $\mathcal{C}^p$ submanifold'' in Corollary 9.3.10 (\cite{bochnakcosteroy1998}, p. 227), it follows that $X$ has a definable $\mathcal{C}^p$ atlas. Therefore, $X$ is an abstract definable $\mathcal{C}^p$ manifold of dimension $m$. 

\begin{lemma}\label{3s}
Let $M$ be an abstract definable $\mathcal{C}^p$ manifold of dimension $m$ and let $f\colon M\to R^n$ be an abstract definable $\mathcal{C}^p$ embedding. Then $f(M)$ is a definable $\mathcal{C}^p$ submanifold of $R^q$ of dimension $q$ and the map $f\colon M\to f(M)$ is an abstract definable $\mathcal{C}^p$ diffeomorphism.
\end{lemma}
\begin{proof}
Straightforward from Lemma \ref{2s} and the fact that $f$ is a homeomorphism between $M$ and $f(M)$.
\end{proof}

\begin{lemma}[Inverse function theorem for abstract definable manifolds]\label{4s}
Let $f\colon M\to N$ be an abstract definable $\mathcal{C}^p$ map, where $M$ and $N$ are abstract definable $\mathcal{C}^p$ manifolds. If $x\in M$ is a point such that $d_x f$ is a linear isomorphism (in particular, $\dim M =\dim N$), then there exist abstract definable open neighborhoods $U_0$ of $x$ and $V_0$ of $f(x)$ such that $f|_{U_0}\colon U_0\to V_0$ is an abstract definable $\mathcal{C}^p$ diffeomorphism. 
\end{lemma}
\begin{proof}
Since $f$ is particularly abstract definable $\mathcal{C}^0$, there exist charts $(U,\phi)$ on $M$ and $(V,\psi)$ on $N$ with $x\in U$ and $f(x)\in f(U)\subseteq V$ such that $\widetilde{f}\mathrel{\mathop:}=\psi \circ f\circ \phi^{-1}\colon \phi(U)\to \psi(V)$ is definable $\mathcal{C}^p$ (see the proof of Lemma \ref{2s}). From the chain rule and the fact that $d_x f$ is an isomorphism, it follows that the map $d_{\phi(x)}\widetilde{f}=d_{f(x)}\psi \circ d_x f\circ d_{\phi(x)} \phi^{-1}$ is a linear isomorphism on $R^m$ with $m\mathrel{\mathop:}=\dim M$. By virtue of the definable inverse function theorem (see for instance Theorem 7.2.11, \cite{vddries98}), there exist definable open neighborhoods $\widetilde{U}\subseteq \phi(U)$ of $\phi(x)$ and $\widetilde{V}\subseteq \psi(V)$ of $\widetilde{f}(\phi(x))$ such that $\widetilde{f}|_{\widetilde{U}}\colon \widetilde{U}\to \widetilde{V}$ is a definable $\mathcal{C}^p$ diffeomorphism. It thus suffices to put $U_0\mathrel{\mathop:}=\phi^{-1}(\widetilde{U})$ and $V_0\mathrel{\mathop:}=\psi^{-1}(\widetilde{V})$, and to observe that $f|_{U_0}\colon U_0\to V_0$ can be given by the composition of abstract definable $\mathcal{C}^p$ diffeomorphisms $\psi^{-1}|_{\widetilde{V}}\circ \widetilde{f}|_{\widetilde{U}}\circ \phi|_{U_0}$.
\end{proof}

\begin{lemma}[Theorem 1.11, \cite{escribano2002}]\label{5s}
For $p>0$, any definable $\mathcal{C}^p$ submanifold of $R^n$ is definably $\mathcal{C}^p$ diffeomorphic to a definable $\mathcal{C}^{p+1}$ submanifold. Two definably $\mathcal{C}^p$ diffeomorphic definable $\mathcal{C}^{p+1}$ submanifolds of $R^n$ are definably $\mathcal{C}^{p+1}$ diffeomorphic. 
\end{lemma}

\begin{theorem}\label{smo}
Any abstract definable $\mathcal{C}\sp{p}$ manifold has a compatible $\mathcal{C}\sp{p+1}$ atlas.
\end{theorem}
\begin{proof}
In view of Theorem \ref{wet}, for any abstract definable $\mathcal{C}^p$ manifold $M$ of dimension $m$ there exists an abstract-definable $\mathcal{C}^p$ embedding $g\colon M\to R^n$, for some $n\in \mathbb{N}$. Since $g(M)$ is a definable $\mathcal{C}^p$ submanifold of $R^n$ of dimension $m$ (see Lemma \ref{3s}), we have a definable $\mathcal{C}^p$ diffeomorphism $f\colon g(M)\to N$ where $N$ is a definable $\mathcal{C}^{p+1}$ submanifold of $R^n$ of dimension $m$, by Lemma \ref{5s}. Pick a finite definable $\mathcal{C}^{p+1}$ atlas $\mathcal{B}$ over $N$. For each chart $(V,\psi)$ in $\mathcal{B}$, put $\widetilde{V}\mathrel{\mathop:}=(f\circ g)^{-1}(V)$ and $\widetilde{\psi}\mathrel{\mathop:}=\psi\circ f\circ g\colon \widetilde{V}\to \psi(V)\subseteq R^m$. It is not hard to see that $\widetilde{\mathcal{B}}\mathrel{\mathop:}=\{(\widetilde{V},\widetilde{\psi})\,:\, (V,\psi)\in \mathcal{B}\}$ is an abstract-definable $\mathcal{C}^{p+1}$ atlas on $M$ of dimension $m$. Moreover, given any chart $(U,\phi)$ in the initial abstract definable $\mathcal{C}^p$ atlas on $M$ and any chart $(\widetilde{V}, \widetilde{\psi})\in \widetilde{\mathcal{B}}$ with $U\cap \widetilde{V}\neq \emptyset$ it follows that on $\phi(U\cap \widetilde{V})$ the map $\widetilde{\psi}\circ \phi^{-1}=(\psi\circ f\circ g)\circ \phi^{-1}=(\psi\circ f)\circ (g\circ \phi^{-1})$ is definable $\mathcal{C}^p$. On the other hand, on $\widetilde{\psi}(U\cap \widetilde{V})$ the map $\phi\circ \widetilde{\psi}^{-1}=\phi\circ (g^{-1}\circ f^{-1}\circ \psi^{-1})=(\phi\circ g^{-1})\circ (f^{-1}\circ \psi^{-1})$ is definable, and since by Lemma \ref{4s}  the map $g$ is a local abstract definable $\mathcal{C}^p$ diffeomorphism, $\phi\circ \widetilde{\psi}^{-1}$ is also of class $\mathcal{C}^p$. Therefore, $\widetilde{\mathcal{B}}$ is $\mathcal{C}^p$-compatible with the fixed abstract definable $\mathcal{C}^p$ atlas on $M$.
\end{proof}

\end{document}